\documentclass[11pt,a4paper]{article}

\usepackage{mathrsfs, mathtools}
\DeclarePairedDelimiterX\setc[2]{\{}{\}}{\,#1 \;\delimsize\vert\; #2\,}
\usepackage{amsthm}
\usepackage[english]{babel}
\usepackage{amsfonts}
\usepackage[psamsfonts]{eucal}
\usepackage[T1]{fontenc}
\usepackage[utf8]{inputenc}
\usepackage{amssymb,amsbsy,amsmath}
\numberwithin{equation}{section} 
\newtheorem{theorem}{Theorem}[section]
\newtheorem{corollary}{Corollary}[theorem]
\newtheorem{lemma}[theorem]{Lemma} 
\newtheorem{assumption}[theorem]{Assumption} 
 
\newtheorem{proposition}[theorem]{Proposition} 
\newtheorem{remark}[theorem]{Remark} 
 
\usepackage[round]{natbib}
\usepackage{url}
\usepackage{hyperref}
\usepackage{verbatim}
\usepackage{graphicx,scalerel}

\usepackage{leftidx}
\mathtoolsset{showonlyrefs}
\usepackage{booktabs}
\usepackage{color}
\usepackage{amssymb}
\usepackage{makeidx}
\usepackage{lmodern}
\usepackage{float}
\usepackage{physics}
\usepackage{multirow}
\usepackage{xcolor}
\usepackage{bm}
\usepackage{bbm}
\usepackage{romannum}
\AtBeginDocument{\pagenumbering{arabic}}

\usepackage{caption}
\def\keyname{\textbf{Keywords}\enspace}
\def\key#1{\par\addvspace\medskipamount{\rightskip=0pt plus1cm
\def\and{\ifhmode\unskip\nobreak\fi\ $\cdot$
}\noindent\keyname\ignorespaces#1\par}}

\def\JELname{\textbf{JEL Classification}\enspace}
\def\JEL#1{\par\addvspace\medskipamount{\rightskip=0pt plus1cm
\def\and{\ifhmode\unskip\nobreak\fi\ $\cdot$
}\noindent\JELname\ignorespaces#1\par}}

\def\MSCname{\textbf{2010 MSC}\enspace}
\def\MSC#1{\par\addvspace\medskipamount{\rightskip=0pt plus1cm
\def\and{\ifhmode\unskip\nobreak\fi\ $\cdot$
}\noindent\MSCname\ignorespaces#1\par}}

\def \B {\mathcal B}
\def \F {\mathcal F}

\def \bP {\mathcal P}
\def \R {\mathcal R}

\def \bF {\mathbb F}

\def \bH {\mathbb H}
\def \bE {\mathbb E}
\def \bL {\mathbb L}
\def \bN {\mathbb N}
\def \P {\mathbb P}
\def \bQ {\mathbb Q}
\def \bR {\mathbb R}
\def \bS {\mathbb S}

\newcommand{\ud}{\mathrm d}

\allowdisplaybreaks

\title{A change of measure formula for recursive conditional expectations}

\author{Luca Di Persio \thanks{Department of Computer Science, University of Verona, via Ca' Vignal 2, 37129 Verona, Italy. E-mail: luca.dipersio@univr.it} \and Alessandro Gnoatto \thanks{(Corresponding Author) Department of Economics, University of Verona, via Cantarane 24, 37129 Verona, Italy. E-mail: alessandro.gnoatto@univr.it} \and Marco Patacca\thanks{Department of Economics, University of Verona, via Cantarane 24, 37129 Verona, Italy. E-mail: marco.patacca@univr.it}}
\date{\today}

\begin{document}
\maketitle

\begin{abstract}
We derive a representation for the value process associated to the solutions of forward-backward stochastic differential equations in a jump-diffusion setting under multiple probability measures. Motivated by concrete financial problems, the latter representations are then applied to devise a generalization of the change of numéraire technique allowing to obtain recursive pricing formulas in the presence of non-linear funding terms due to e.g. collateralization agreements.
	 \key{Pricing \and Change of measure \and BSDE \and recursive conditional expectation \and non-linear valuation }
 \JEL{C63 \and G12 \and G13}
\end{abstract}

\section{Introduction and Motivation}
Let $\P$ and $\bQ$ be two probability measures on the measurable space $\left( \Omega,\F\right)$. From the Radon-Nikodym theorem, it is well known that $\bQ\ll\P$ on $\left( \Omega,\F\right)$ if and only if there exists an $\F$-measurable and non negative random variable $ \mathcal{E} : \Omega \rightarrow \bR^+$ with $\bE^{\P}\left[ \mathcal{E}\right]=1$, such that, for any random variable $X\in L^1\left(\Omega,\F,\bQ\right)$ we have
\begin{equation}\label{eq:RadNyk}
    \bE^{\bQ}\left[X\right]=\bE^{\P}\left[ \mathcal{E} X\right] \; , \; \bQ - a.s.
\end{equation}
The previous result can be generalized to conditional expectations,  which is sometimes referred to as the ``abstract Bayes rule''. 
More specifically, if $\mathcal{G}$ is a $\sigma$-algebra with $\mathcal{G}\subseteq\F$ we have
\begin{equation}\label{eq:bayesRule}
        \bE^{\bQ}\left[X|\mathcal{G}\right]=\frac{\bE^{\P}\left[ \mathcal{E} X|\mathcal{G}\right]}{\bE^{\P}\left[ \mathcal{E}|\mathcal{G}\right]} \; , \; \bQ - a.s.
\end{equation}
The above relations play a relevant role in several applications, particularly in Mathematical Finance, see e.g. \cite{bjork2020}. Within this setting, the likelihood ratio $ \mathcal{E}$, due to the Girsanov theorem, is usually given by the Doléans-Dade stochastic exponential driven by possibly discontinuous semimartingales. Such tools are well understood for standard conditional expectations, while the case when the conditional expectation has a recursive structure is less clear. For the sake of clarity, let us first introduce such situation in an informal way. 

By recursive expectation, we mean an expression of the following form
\begin{equation}\label{eq:condExp}
    Y_t=\bE^{\bQ}\left[ g(X_T)+\int_t^T f(s,X_s,Y_s,Z_s)\ud s|\F_t \right],
\end{equation}
where the functions $g,f$ and the processes $X,Y$ and $Z$ are {\it regular enough}, also having implicitly introduced the filtration $\bF=(\F_t)_{0\leq t\leq T}$, $T>0$  being a fixed time horizon. Such filtration is generated by the process $X\coloneqq\lbrace X_t, t \in [0,T]\rbrace$, typically assumed to be an It\^{o} semimartingale. 

Expectations like \eqref{eq:condExp} naturally arise from the study of Forward-Backward Stochastic Differential Equations (FBSDEs) of the form
\begin{align}\label{eq:fbsdes}
    \ud X_t & = \mu(t,X_t)\ud t + \sigma(t,X_t)\ud W_t^{\bQ} \quad , \quad X_0 = x \\
    -\ud Y_t & = f(t,X_t,Y_t,Z_t)\ud t - Z_t\ud Z_t^{\bQ} \quad , \quad Y_T=g(X_T) \, ,
\end{align}
where all vector field functions are assumed to be {\it sufficiently regular}, while $W^{\bQ}$ is a standard Brownian motion under the reference measure $\bQ$. For a background on FBSDEs, we refer the reader to \cite{bismut1973}, \cite{pardouxPeng1990}, \cite{el1997backward}, \cite{maYong1999} and \cite{delong2013backward} among others.

The aim of this work is to provide  a change of measure formula for conditional expectations of the form \eqref{eq:condExp} associated to FBSDEs of the form \eqref{eq:fbsdes} within a Markovian Jump Diffusion setting. We derive such result in Theorem \ref{th:mainTh} which constitutes the main result of the present paper.

Our motivation to consider this type of recursive conditional expectations is mainly steered by financial applications. As a response to the financial crisis, the valuation of contingent claim has been generalized along several directions.
For example \citet{piterbarg2010funding} obtained, starting from a Black-Scholes type replication argument and applying the Feynman-Ka\v{c} formula, the following recursive expression for the price of a contingent claim when collateral is exchanged:
\begin{equation}\label{eq:piterbarg}
        Y_t=\bE^{\bQ}\left[ e^{-\int_t^T r_u^c\ud u} Y_T-\int_t^T e^{-\int_t^T r_s^c\ud s} (r_u^f-r_u^c)(C_u-Y_u)\ud u\big|\F_t\right]\,,
\end{equation}
where the collateral $C$ and the short rates $r^c,r^f$ are sufficiently regular, such that the above expression is well posed: in applications, $C$ is assumed to be a Lipschitz function of $Y$.
Expressions of the aforementioned form have been then justified via BSDE-techniques in, e.g.,  \citet{bielecki2015valuation}, \citet{bichuch2018arbitrage}, \cite{BRIGO2019788}, \citet{biagini2021unified} and \cite{gnoattoSeiffert2021}.
Such recursive expectations are now well understood, and are at least approximately implemented 
in several industry-standard software packages via a combination of quadrature formulas for the time integral, and Monte Carlo simulations for the expectation. 
To the best of our knowledge, a comprehensive study of the change of measure/change of numéraire technique applied to such conditional expectations, is still missing. 
It is worth mentioning that in \cite{cohen2015stochastic} the authors observed that linear BSDEs can be used to represent a measure change, also providing the formulation of a linear BSDE under two probability measures. The conditional expectation they obtained does not exhibit a recursive structure since the integrand is a predictable process not depending on the value process nor on the controls. We take this idea one step further. Indeed, we consider a fairly general structure of the driver which is relevant for the aforementioned applications, allowing us to generalize the ``change of numéraire'' technique, see \citet{geman1995changes}, so as to also cover the recursive conditional expectation case.\\

The paper is organised as follows: in Section \ref{sec:NotationAndSetting} we present the probabilistic setup, whereas in Section \ref{sec:bsdeJumps} we derive  our main result, i.e. Theorem \ref{th:mainTh}, subsequently, in Section \ref{sec:FinancialApplications}, we provide some financial applications. 
Concerning these latter, we exploit our Theorem \ref{th:mainTh} to analyse the market model introduced in \citet{bielecki2015valuation} focusing on the change from the physical to the risk-neutral measure, allowing us also to emphasize the  link between our results and the benchmark approach proposed by \citet{platen2006benchmark}. We also show how our method can be used to revisit the valuation of exchange options in a setting including collateral, permitting us to deduce a general procedure to perform changes of numéraire in non-linear market models. A simple example of market driven by a jump process is also provided.

\section{Notations and setting}\label{sec:NotationAndSetting}
Let $T\in\bR_+$,  $T<\infty$, be a fixed horizon time. We consider a filtered probability space $(\Omega,\F,\bF,\P)$, where the filtration $\bF\coloneqq\lbrace \F_t, t\in \left[0,T\right]\rbrace$ is assumed to be complete and right continuous. For all semimartingales considered in the sequel, we consider a version with right continuous paths and left limits.

Let $d\in \bN^+$. We assume that the filtered probability space supports an $\bR^d$-valued Brownian motion $W^\P\coloneqq \lbrace W_t^\P,t\in[0,T]\rbrace$, for a finite positive $T$, and a Poisson random measure $N$, with compensator $\vartheta(\ud t, \ud z)$ on $\bR^d\setminus\lbrace 0\rbrace$, while $\vartheta$ is $\sigma$-finite and such that $\int_0^T\int_{\bR^d\setminus\lbrace 0\rbrace}(1\wedge|z|^2)\vartheta(\ud t, \ud z) <\infty$. Moreover, we define
the compensated random measure $\tilde{N}$, associated to $N$, as follows 
\begin{equation}
    \tilde{N}(\ud t, \ud z) \coloneqq N(\ud t, \ud z)- \vartheta(\ud t, \ud z),
\end{equation}
and we introduce the following space of processes
\begin{itemize}
    \item $\bP$: the predictable $\sigma$-field on $\Omega \times \left[0,T\right]$ generated by all left continuous and $\bF$-adapted processes;
    \item $\bL_T^2(\bR^d)$: the space of all $\F_T$-measurable random variables $X:\Omega\rightarrow \bR^d$ satisfying $||X||^2\coloneqq\bE^{\P}(|X|^2)<+\infty$;
    \item $\bH_T^2(\bR^d)$: the space of all predictable process $\phi:\Omega \times \left[0,T\right]\rightarrow\bR^d$ such that $||\phi||^2\coloneqq\bE^{\P}\left[ \int_0^T|\phi_t|^2\ud t\right]<+\infty$.
\end{itemize}
As previously mentioned, the filtration $\bF$ supports both the Brownian motion $W$ and the Poisson random measure $N$. 
Let us provide the statement of the predictable representation property in the present setting.

\begin{lemma}\label{lem:repProp}
{[}See Lemma \Romannum{3}, 4.24 in \citet{jacod2003limit}{]} Any $\bF$-local martingale $M$ has the representation
\begin{equation}
    M_t = M_0 + \int_0^tZ_s\ud W_s^\P + \int_0^t\int_{\bR^d}U_s(z)\tilde{N}(\ud s,\ud z) \; , \; 0\leq t\leq T \; ,
\end{equation}
where $Z$ and $U$ are $\F$-predictable processes integrable with respect to both $W$ and $\tilde{N}$.
\end{lemma}
For the sake of completeness, let us also further
specify the structure of the random measure:
\begin{assumption}\label{as:1}
$N$ is an integer-valued random measure with compensator 
\begin{equation}
    \vartheta(\ud t, \ud z) = Q(t, \ud z)\eta(t) \ud t
\end{equation}
where $\eta:\Omega\times [0,T]\rightarrow[0,\infty)^d$ is a predictable process and $Q$ is a kernel from $(\Omega\times [0,T],\bP)$ into $(\bR^d,\B(\bR^d))$ satisfying
\begin{equation}
    \int_0^T\int_{\bR^d}|z|^2Q(t,\ud z)\eta(t)\ud t <\infty \; ,
\end{equation}
and we set $N(\lbrace 0\rbrace ,\bR^d)=N((0,T],\lbrace 0 \rbrace ) = \vartheta ((0,T],\lbrace 0 \rbrace ) = 0$.
\end{assumption}

Following \cite{delong2013backward}, we then introduce the standard spaces that allow us to properly define the concept of solution to FBSDEs with jumps. We define
\begin{itemize}
    \item $\bH_{T,N}^2(\bR^d)$, the space of all predictable process $U:\Omega \times \left[0,T\right]\times\bR^d\rightarrow\bR^d$ satisfying $\bE^{\P}\left[ \int_0^T\int_{\bR^d}|U_t(z)|^2Q(t,\ud z)\eta(t)\ud t\right]<+\infty$, the integral being considered w.r.t. to the aforementioned predictable compensator $\tilde{N}$;
    \item $\bS_{T}^2(\bR^d)$ the space of $\bF$-adapted càdlàg processes $Y$, $Y:\Omega\times [0,T]\rightarrow\bR^d$ satisfying $\bE^{\P}\left[\sup\limits_{t\in [o,T]}|Y_t|^2\right]<+\infty$.
\end{itemize}

\section{BSDE with Jumps}\label{sec:bsdeJumps}
Let us outline the FBSDE we consider in the present work.

\subsection{Forward SDE}
Let $T \in \mathbb{R}^+$, $T<\infty$. We consider the following $\bR^d$-valued forward SDE
\begin{multline}\label{eq:forSDE}
    X_s = x+\int_t^s\mu(X_{r-})\ud r + \int_t^s\sigma(X_{r-})\ud W_r^\P \\
    + \int_s^t\int_{\bR^d}\gamma(x_{r-},z)\tilde{N}(\ud r,\ud z) \; , \; 0\leq t\leq s\leq T \; ,
\end{multline}
meaning that we are considering the following system of $\bR$-valued  SDEs
\begin{multline}
    X_s^i = x^i+\int_t^s\mu^i(X_{r-})\ud r + \int_t^s\sigma^i(X_{r-})\ud W_r^\P \\
    + \int_s^t\int_{\bR^d}\gamma^i(x_{r-},z)\tilde{N}(\ud r,\ud z) \; , \quad 0\leq t\leq s\leq T \; , \; 1\leq i\leq d .
\end{multline}
Here, the mappings $\mu^i:\bR^d\rightarrow \bR$, $\sigma^i:\bR^d\rightarrow \bR$, $\gamma^i:\bR^d\times\bR^d\rightarrow \bR$ are all assumed to be measurable for any $i \in \{1, \ldots,n\}$.

Moreover, for any $i \in \{1, \ldots,n\}$ and $K$ denoting, for better readability, a {\it universal} constant such that all below definitions are well posed, we assume that:
\begin{itemize}
    \item[\textbf{(A1)}] the functions $\mu^i$ and $\sigma^i$ are Lipschitz continuous;
    \item[\textbf{(A2)}] the functions $\gamma^i$ are measurable and satisfy 
    \begin{align*}
        |\gamma^i(x,z)|&\leq K(1\wedge |z|) \; , \; (x,z) \in \bR^d\times\bR^d \\
        |\gamma^i(x,z)-\gamma^i(x',z)|&\leq K|x-x'|(1\wedge |z|) \; , \; (x,z), (x',z) \in \bR^d\times\bR^d ;\\
    \end{align*}
    \item[\textbf{(A3)}] $N$ indicates a Poisson random measure generated by a Lévy process with Lévy measure $\nu$, namely:
    $\nu(\{0\})=0$ and $\int_{\bR^d}1\wedge |z|^2 \nu(\ud z)<\infty$.
\end{itemize}
It is worth noting that
Assumption (A3) allows us to simplify Assumption \ref{as:1}. Indeed, setting: 
$\eta(t)\equiv1$ and $Q(t,\ud z)=\nu^{\P}(\ud z)$ the jumps are those of a time-homogeneous Lévy process.
Furthermore, we denote with $(X_s^{t,x})_{t\leq s\leq T}$ the solution to \eqref{eq:forSDE} with initial state $(t,x)$.

\begin{theorem}\label{th:solExSDE}
Under Assumptions \textbf{(A1)}-\textbf{(A3)}
\begin{itemize}
    \item[a)] For each $(t, x)\in [0,T]\times\bR^d$ there exists a unique adapted, càdlàg solution $X^{t,x}\coloneqq (X_s^{t,x})_{t\leq s\leq T}$ to \eqref{eq:forSDE}.
    \item[b)] The solution $X^{t,x}$ is a homogeneous Markov process.
\end{itemize}
\end{theorem}
\begin{proof}
Case a), resp. case b), follows from Theorem 6.2.9., resp. from Theorem 6.4.6, in \citet{applebaum2004levy}.
\end{proof}

\subsection{Related BSDE}\label{sec:relBSDE}
Let us consider the following backward SDE (BSDE) characterized  by a terminal condition and a generator depending on the state process solving the FSDE: \eqref{eq:forSDE}
\begin{multline}\label{eq:BSDE}
    Y_s^{t,x} = g(X_T^{t,x}) + \int_s^T f\left(r,X_{r-}^{t,x},Y_{r-}^{t,x},Z_{r}^{t,x},\int_{\bR^d} U_{r}^{t,x}( z)\delta( z)\nu^{\P}(\ud z)\right)\ud r \\
    -\int_s^T\alpha(X_r^{t,x})Y_r^{t,x}\ud r +\int_s^T\beta(X_r^{t,x})Z_r^{t,x}\ud r  + \int_s^T\int_{\bR^d}U_r^{t,x}( z)\delta( z)\nu^{\P}(\ud z)\ud r \\ - 
    \int_s^T Z_{r}^{t,x}\ud W_r^{\P} - \int_s^T\int_{\bR^d}U_r^{t,x}( z)\tilde{N}(\ud r, \ud z).
\end{multline}
Again considering a suitable constant $K>0$, we assume that the function:
\begin{itemize}
    \item[\textbf{(A4)}] $f:[0,T]\times \bR^d\times \bR\times \bR^d\times \bR^d \rightarrow \bR$ is Lipschitz continuous, i.e.
    \begin{multline}
        |f(t,x,y,z,u)-f'(t,x',y',z',u')|\\
        \leq K\left(|x-x'|+|y-y'|+|z-z'|+|u-u'| \right)
    \end{multline}
    for all $(t,x,y,z,u),(t,x',y',z',u')\in [0,T]\times \bR^d\times \bR\times \bR^d\times \bR^d$,  uniformly in $t$;
    \item[\textbf{(A5)}] $g:\bR^d\rightarrow\bR$ is measurable and satisfies
    \begin{equation}
        |g(x)-g(x')|\leq K\left(|x-x'| \right) \; , \; x,x'\in\bR^d;
    \end{equation}    
    \item[\textbf{(A6)}]  $\alpha:\bR^d\rightarrow\bR$ satisfies $|\alpha(X_r^{t,x})|\leq K,\; t\leq r\leq T$;
    \item[\textbf{(A7)}] $\beta:\bR^d\rightarrow\bR$ satisfies $|\beta(X_r^{t,x})|\leq K,\; t\leq r\leq T$;    
    \item[\textbf{(A8)}] $\delta:\bR^d\rightarrow\bR$ satisfies $\int_{\bR^d}|\delta( z)|^2Q^{\P}(t,\ud z)\leq K,\; t\leq r\leq T$  and  $\delta(z)>-1$,  $z \in \bR^d$.
\end{itemize}

\begin{theorem}\label{th:uniqSol}
{[}See Theorem 4.1.3 in \citet{delong2013backward}{]} Under assumptions (A1)–(A5), there exists a unique solution $((X_s^{t,x}),(Y_s^{t,x}),(Z_s^{t,x}),(U_s^{t,x}))\in \bS_{T}^2(\bR^d)\times \bS_{T}^2(\bR)\times \bH_{T}^2(\bR^d)\times \bH_{T,N}^2(\bR^d)$ to the forward-backward SDE (FBSDE) \eqref{eq:forSDE}, \eqref{eq:BSDE}.   
\end{theorem}
In what follows we provide the main result of the paper, namely a representation for the value process of the FBSDE (\ref{eq:forSDE} - \ref{eq:BSDE}) under different probability measures.

\begin{theorem}\label{th:mainTh}
Under the preceeding assumptions, let $ T \in \mathbb{R}^+$, $T<\infty$ and consider $0 <t\leq T$. Let us define
\begin{align}
\label{eq:densityMainTh}
\begin{aligned}
    \mathcal E_t^T &\coloneqq \exp\left\lbrace-\frac{1}{2} \int_t^T \beta(X_r^{t,x})^2\ud r + \int_t^T \beta(X_r^{t,x})\ud W^{\P} \right. \\
    &\quad\left. +\int_t^T\int_{\bR^d}\ln(1+\delta( z))-\delta( z)\nu^{\P}(\ud  z) \ud r + \int_t^T\int_{\bR^d}\ln(1+\delta( z))\tilde{N}(\ud r,\ud z) \right\rbrace\\
    &=\frac{H_T}{H_t}  \, ,
    \end{aligned}
\end{align}

\noindent i.e. $H= \lbrace H_r^\P,r\in[t,T]\rbrace$ solves the FSDE
\begin{equation}
    \ud H_r = H_r\left(\beta(X_r^{t,x})\ud W_r^{\P} + \int_{\bR^d}\delta( z)\tilde{N}(\ud r,\ud  z) \right) \; , \; 0<t\leq r\leq T \;,
\end{equation}

\noindent then the process $Y^{t,x}$ in \eqref{eq:BSDE} admits the representation
\begin{multline}
    Y_t^{t,x} = \bE^{\P} \left[ g(X_T^{t,x}) e^{-\int_t^T \alpha(X_r^{t,x})\ud r} \mathcal E_t^T + \int_t^T e^{-\int_t^r \alpha(X_s^{t,x})\ud s}\mathcal E_t^r \right. \\
    \left. \cdot f\left(r,X_{r-}^{t,x},Y_{r-}^{t,x},Z_{r}^{t,x},\int_{\bR^d} U_{r}^{t,x}( z)\delta( z)\nu^{\P}(\ud z)\right)      \ud r    \, \big|\, \F_t\right] \; \P \,\mbox{a.s.}\, 
\end{multline}

\noindent and, $\mathcal E^T$ defined in \eqref{eq:densityMainTh} is a true $(\P,\F)$-martingale. \\

\noindent Define
\begin{itemize}
    \item $\frac{\partial \P^2}{\partial \P}\bigg{|}_{\F_r}\coloneqq \mathcal E_t^r \; ;$
    \item $W_r^{\P^2} \coloneqq W_r^{\P} - \int_t^r \beta(X_s^{t,x}))\ud s  \; ;$
    \item $\nu^{\P^2}(\ud z) \coloneqq (1+\delta( z))\nu^{\P}(\ud z)  \; ,$
\end{itemize}

\noindent then the above defined FBSDE, resp. the process $Y_s^{t,x}$, admits the equivalent representation 
\begin{multline}
    Y_s^{t,x} = g(X_T^{t,x}) + \int_s^T f\left(r,X_{r-}^{t,x},Y_{r-}^{t,x},Z_{r}^{t,x},\int_{\bR^d} U_{r}^{t,x}( z) \frac{\delta( z)}{1+\delta( z)}\nu^{\P^2}(\ud z)\right) \ud r \\
    -\int_s^T\alpha(X_r^{t,x})Y_r^{t,x}\ud r + \int_s^T\int_{\bR^d}U_r^{t,x}( z)\frac{\delta( z)}{1+\delta( z)}\nu^{\P^2}(\ud z)\ud r \\ - 
    \int_s^T Z_{r}^{t,x}\ud W_r^{\P^2} - \int_s^T\int_{\bR^d}U_r^{t,x}( z)\tilde{N}(\ud r, \ud z)\; ,
\end{multline}
resp.
\begin{multline}\label{eq:mainThEqChange}
    Y_t^{t,x} = \bE^{\P^2} \left[ g(X_T^{t,x}) e^{-\int_t^T \alpha(X_r^{t,x})\ud r} + \int_t^T e^{-\int_t^r \alpha(X_s^{t,x})\ud s} \right. \\
    \left. \cdot f\left(r,X_{r-}^{t,x},Y_{r-}^{t,x},Z_{r}^{t,x},\int_{\bR^d} U_{r}^{t,x}( z) \frac{\delta( z)}{1+\delta( z)}\nu^{\P^2}(\ud z)\right) \ud r    \, \big|\, \F_t\right] \; \P^2 \,\mbox{a.s.}\, .
\end{multline}
under the probability measure $\P^2$,.
\end{theorem}

\begin{proof}
Thanks to Assumptions \textbf{(A7)} and \textbf{(A8)}, the process $\mathcal E^T$ is a true $(\P,\F)$-martingale by Proposition 2.5.1 in \citet{delong2013backward}.
Let us define $B_r\coloneqq \exp{\int_0^r\alpha(X_s^{t,x})\ud s}$. We consider the square integrable martingale 
\begin{align*}
M_s=\bE^{\P}\left[ H_T\frac{g(X_T^{t,x})}{B_T}+\int_t^T \frac{H_r}{B_r}f\left(r,X_{r-}^{t,x},Y_{r-}^{t,x},Z_{r}^{t,x},\int_{\bR^d} U_{r}^{t,x}( z)\delta( z)\nu^{\P}(\ud z)\right)\ud r \, \big|\, \F_s \right],
\end{align*}
for $t\leq s\leq T$. The martingale representation theorem implies the existence of stochastic processes $\tilde{Z}\in \bH_{T}^2(\bR^d)$, $\tilde{U}\in \bH_{T,N}^2(\bR^d)$, such that
\begin{align}
M_s=M_t    + \int_t^s\tilde{Z}_r^{t,x}\ud W_r^{\P} + \int_t^s\int_{\bR^d} \tilde{U}_r^{t,x}( z)\tilde{N}(\ud r, \ud z) \; .
\end{align}
Setting $\tilde{Z}_r^{t,x}\coloneqq \frac{H_r}{B_r}\left(Z_r^{t,x}+Y_r^{t,x}\beta (X_r^{t,x}) \right)$  and $\tilde{U}_r^{t,x}( z)\coloneqq \frac{H_r}{B_r}\left(U_r^{t,x}( z)(1+\delta( z))+Y_r^{t,x}\delta( z) \right)$, the It\^{o}-D{\"o}blin lemma implies
\begin{align}
    \ud \left(\frac{Y_s^{t,x}}{B_s} \right) = & \frac{1}{H_s}\left(\ud \left(H\frac{Y_s^{t,x}}{B_s}\right)-\frac{Y_s^{t,x}}{B_s}\ud H_s-\ud \left[ H,\frac{Y_s^{t,x}}{B_s}\right] \right)\\
        = & \frac{1}{H_s}\left(-\frac{H_s}{B_s}f\left(s,X_{s-}^{t,x},Y_{s-}^{t,x},Z_s^{t,x},\int_{\bR^d} U_{r}^{t,x}( z)\delta( z)\nu^{\P}(\ud z)\right)\ud s \right.\\
        & +\tilde{Z_s}\ud W_s^{\P} - \frac{Y_s^{t,x}}{B_s}H_s\beta(X_s^{t,x})\ud W_s^{\P} - \frac{Z_s^{t,x}}{B_s}H_s\beta(X_s^{t,x})\ud s \\
        & + \int_{\bR^d}\frac{H_s}{B_s}\left( U_s^{t,x}( z)(1+\delta( z))+Y_s^{t,x}\delta( z)\right)\tilde{N}(\ud s, \ud z)\\
        & - \int_{\bR^d}\frac{Y_s^{t,x}}{B_s}H_s\delta( z)\tilde{N}(\ud s, \ud z)\\
         & \left. - \int_{\bR^d}\frac{1}{B_{s-}}H_{s-}\delta( z)U_{s-}^{t,x}N(\ud s, \ud z) \pm \int_{\bR^d}\frac{1}{B_{s-}}H_{s-}\delta( z)U_{s-}^{t,x}\nu^{\P}(\ud z)\ud s\right)\\
   =&  -\frac{1}{B_s} f\left(s,X_{s-}^{t,x},Y_{s-}^{t,x},Z_s^{t,x},\int_{\bR^d} U_{s}^{t,x}( z)\delta( z)\nu^{\P}(\ud z)\right)\ud s + \frac{Z_s^{t,x}}{B_s}\ud W_s^{\P} \\
    &- \frac{Z_s^{t,x}}{B_s}\beta(X_s^{t,x})\ud s    - \frac{1}{B_s}\int_{\bR^d}\delta( z)U_{s}^{t,x}( z)\nu^{\P}(\ud z)\ud s + \int_{\bR^d}\frac{1}{B_s}U_s^{t,x}( z)\tilde{N}(\ud s, \ud z),
\end{align}
and since we also have
\begin{equation}
    \ud \left(\frac{Y_s^{t,x}}{B_s}\right) = \frac{1}{B_s}\ud Y_s^{t,x}-\frac{Y_s^{t,x}}{B_s^2}\ud B_s,
\end{equation}
then
\begin{align}
    \ud Y_s^{t,x} & = B_s\ud \left(\frac{Y_s^{t,x}}{B_s}\right)+\frac{Y_s^{t,x}}{B_s}\ud B_s\\
    = &  -f\left(s,X_{s-}^{t,x},Y_{s-}^{t,x},Z_s^{t,x},\int_{\bR^d} U_{s}^{t,x}( z)\delta( z)\nu^{\P}(\ud z)\right)\ud s\\
    & + Z_s^{t,x}\ud W_s^{\P} -Z_s^{t,x}\beta(X_s^{t,x})\ud s\\
    & -\int_{\bR^d}\delta ( z) U_s^{t,x}( z)\nu^{\P}(\ud  z)\ud s +\int_{\bR^d}U_s^{t,x}( z)\tilde{N}(\ud s, \ud z) + \alpha (X_s^{t,x})Y_s^{t,x}\ud s\; \,.
\end{align}
Summing up the previous points we recover
\eqref{eq:BSDE}, namely:
\begin{multline}
    -\ud Y_s^{t,x} = \left(f\left(s,X_{s-}^{t,x},Y_{s-}^{t,x},Z_s^{t,x},\int_{\bR^d} U_{s}^{t,x}( z)\delta( z)\nu^{\P}(\ud z)\right) + Z_s^{t,x}\beta(X_s^{t,x})   \right.\\
    \left.+ \int_{\bR^d}\delta( z)U_s^{t,x}( z)\nu^{\P}(\ud z) - \alpha(X_s^{t,x})Y_s^{t,x}\right)\ud s - Z_s^{t,x}\ud W_s^{\P} - \int_{\bR^d}U_s^{t,x}( z)\tilde{N}(\ud s, \ud z) \; .
\end{multline}
Finally, since $\mathcal E^T$ is a true $(\P,\F)$-martingale,  we can apply the Girsanov Theorem as in Theorem 2.5.1 in \citet{delong2013backward}. By repeating the previous steps under the equivalent measure $\P^2$, we have that the {\it new} representation of the process $Y_s^{t,x}$ under $\P^2$ reads as follows:
\begin{multline}
    Y_t^{t,x} = \bE^{\P^2} \left[ g(X_T^{t,x}) e^{-\int_t^T \alpha(X_r^{t,x})\ud r} + \int_t^T e^{-\int_t^r \alpha(X_s^{t,x})\ud s} \right. \\
    \left. \cdot f\left(r,X_{r-}^{t,x},Y_{r-}^{t,x},Z_{r}^{t,x},\int_{\bR^d} U_{r}^{t,x}( z) \frac{\delta( z)}{1+\delta( z)}\nu^{\P^2}(\ud z)\right) \ud r    \, \big|\, \F_t\right] \; \P^2 \,\mbox{a.s.}\, .
\end{multline}
\end{proof}

We notice that changing the measure in \eqref{eq:mainThEqChange} results in the appearance of the density process $\mathcal{E}^T$ both next to the terminal condition (which is standard) and inside the integral term which is an uncommon feature of the non-linear and recursive context we are considering.

\section{Financial Applications}\label{sec:FinancialApplications}
In what follows,  we adopt the setup of \citet{bielecki2015valuation}, see also \cite{gnoattoSeiffert2021}. In particular, we fix a positive but finite time horizon $T$. All stochastic processes introduced in the sequel are   defined on a filtered probability space $(\Omega,\F,\bF,\P)$ where the filtration $\bF\coloneqq\lbrace \F_t, t\in \left[0,T\right]\rbrace$ satisfies usual assumptions, while $\F_0$ is the trivial $\sigma-$algebra at initial time. All processes introduced in the sequel are assumed to be $\bF$-adapted and we work with càdlàg versions.

We denote by $S^i$ the price (ex-dividend price) of the i-th risky asset with cumulative dividend stream $A^i$ with $i=1,\ldots,d$. We take the perspective of the hedger, i.e., the entity that sells a certain contract and constructs a hedging portfolio against the liability stemming from the contract. Let $\xi_t^i$ denote the number of the hedger's positions in asset $S^i$ at time $t$. We let $B^0=B$ be the cash account for unsecured funding and introduce $B^1,\ldots,B^d$ cash accounts allowing for the financing of repurchase agreements (repo trading) on the risky assets.
We assume that:
\begin{itemize}
    \item[i)] $S^i$ for $i=1,\ldots,d$ are semimartingales;
    \item[ii)] $A^i$ for $i=1,\ldots,d$ are finite variation processes with $A_0^i=0$;
    \item[iii)] $B^j$ for  $j=1,\ldots,d$ are strictly positive and continuous processes of finite variation with $B_0^j=1$.
\end{itemize}

We introduce the cumulative dividend price $S^{i,cld}$, given by
\begin{equation}
    S_t^{i,cld}\coloneqq S_t^i+B_t^i\int_{(0,t]}(B_u^i)^{-1}\ud A_u^i \; , \quad t\in[0,T] \; ,
\end{equation}
together with its discounted version
\begin{equation}
    \hat{S}_t^{i,cld}\coloneqq \hat{S}_t^i+\int_{(0,t]}(B_u^i)^{-1}\ud A_u^i \; ,
\end{equation}
with $\hat{S}_t^{i,cld}=(B^i)^{-1}S^i$.

A trading strategy, or dynamic portfolio, is composed of positions in risky assets $\xi^i$, $i=1,\ldots,d$, positions in unsecured funding repo accounts $\psi^j$, $i=0,\ldots,d$, and the cash account related to lending and borrowing collateral $\psi^{c,l}$, $\psi^{c,b}$ 
\begin{equation}
    \varphi=(\xi^1,\ldots,\xi^d,\psi^0,\ldots,\psi^d,\psi^{c,l},\psi^{c,b}).
\end{equation}
Processes $(\xi^1,\ldots,\xi^d)$ are assumed to be $\bF$-predictable, while $(\psi^0,\ldots,\psi^d,\psi^{c,l},\psi^{c,b})$ are $\bF$-adapted, allowing the stochastic integrals we then introduce to be well defined. Trading strategies are subject to admissibility constraints. A contract is represented by means of its cumulative stream of cash flows, hence by a process $A$ of finite variation with càdlàg paths, such that $A_{0^-}=0$. As to provide an example, if the hedger sells a standard call option written on $S^i$, with maturity time $T$ and strike price $\mathcal{K}$, at a price $p$, then $A_t=p\mathbbm{1}_{[0,T]}(t)-(S_T^i-\mathcal{K})^+\mathbbm{1}_{[T]}(t)$.

To prevent arbitrage opportunities via trivial positions in the cash accounts, we impose the repo constraint
\begin{equation}\label{eq:repoConst}
    \psi_t^iB_t^i+\xi_t^iS_t^i=0  \; , \; \ud\P \otimes \ud t - a.s.
\end{equation}
meaning that the holdings in the i-th cash account is instrumental to the purpose of financing the position in the i-th risky asset only.

In the setting of \citet{bielecki2015valuation}, the meaning of a martingale measure is that of a measure such that all processes $\hat{S}^{i,cld}$ are local martingales.

\begin{proposition}\label{prop:prop3_1}
{[}See Proposition 3.1 in \citet{bielecki2015valuation}{]} Assume that all strategies available to the hedger are admissible and satisfy the repo constraint \eqref{eq:repoConst}. If $\exists$ a probability measure $\bQ$ on $(\Omega,\F_T)$ such that $\bQ\sim\P$ and the processes $\hat{S}^{i,cld}$, $i=1,\ldots,d$ are local $(\bQ,\bF)$ martingales, then the market model is free of arbitrage opportunities for the hedger.
\end{proposition}
Let us now specialize the above introduced setting by means of the following assumptions:
\begin{enumerate}
    \item \label{as:rates}All processes $B^j$ are absolutely continuous, meaning that they can be written as $\ud B_t^j=r_t^jB_t^j\ud t$ for some $\bF$-adapted processes $r^j$, $j=0,\ldots,d$. To simplify the treatment of existence and uniqueness of BSDEs, we assume that such rates are bounded.
    \item We let $W^{1,\P},\ldots,W^{d,\P}$ be standard mutually independent Brownian motions and we define $W^{\P}=\begin{psmallmatrix} W^{1,\P}\\ \vdots\\ W^{d,\P} \end{psmallmatrix}$.
    \item \label{as:vecFields}We introduce vector fields $\mu^1,\ldots,\mu^d, \sigma^1,\ldots,\sigma^d$ which are assumed to be bounded stochastic process, and we write \begin{equation}
        \mu=\begin{psmallmatrix} \mu^{1}\\ \vdots\\ \mu^{d} \end{psmallmatrix} \; , \;  \Sigma=\begin{psmallmatrix} \sigma^{1}&&\\ &\ddots&\\ &&\sigma^{d} \end{psmallmatrix}  \; .
    \end{equation}
\end{enumerate}
Each asset is then characterized by the following
dividend process
\begin{equation}
    A_t^i=\int_0^t k_u^i S_u^i \ud u \; ,
\end{equation}
where $k^i$, $i=1,\ldots,d$ are $\bF$-adapted bounded stochastic processes,

while the risky assets evolve according to
\begin{equation}\label{eq:Sdyn}
    \ud S = \ud \begin{psmallmatrix} S_t^1\\ \vdots\\ S_t^d\end{psmallmatrix} = diag(S)\left( \mu_t\ud t + \Sigma_t \bar{\rho}_t\ud W_t^\P\right)  \; ,  
\end{equation}
where $\bar{\rho}_t$ is a matrix such that $\rho_t=\bar{\rho}_t\bar{\rho}_t^\intercal$ represents the correlation matrix between  
$W^1,\ldots,W^d$. Moreover, $\bar{\rho}$ is such that $\Sigma_t\bar{\rho}_t$ is invertible, for every $t\in [0,T]$.
Given the previous setting, let us recall the shape of the martingale measure $\bQ$ as stated in Lemma 5.2 in \citet{bielecki2015valuation}.
\begin{lemma}
Under the measure $\bQ$, risky assets evolve according to 
\begin{equation}
    \ud \begin{psmallmatrix} S_t^1\\ \vdots\\ S_t^d \end{psmallmatrix} = diag(S) \left( \begin{psmallmatrix} r_t^1-k_t^1 \\ \vdots \\ r_t^d-k_t^d \end{psmallmatrix}\ud t + \Sigma_t\bar{\rho}_t\ud W_t^\bQ \right) ,
\end{equation}
where the multivariate Brownian motion $W^\bQ$ is given as
\begin{equation}\label{eq:QBrow}
    W^\bQ \coloneqq W_t^\P + \int_0^t\left( \Sigma_u\bar{\rho}_u \right)^{-1}\begin{psmallmatrix} \mu_u^1-r_u^1+k_u^1\\ \vdots \\ \mu_u^d-r_u^d+k_u^d \end{psmallmatrix}\ud u.
\end{equation}
\end{lemma}
Given the assumptions \ref{as:rates} and \ref{as:vecFields} and denoting 
\begin{align}
\label{eq:GirsanovKernelDiffusion}
\theta_u\coloneqq \left( \Sigma_u\bar{\rho}_u \right)^{-1}\begin{psmallmatrix} \mu_u^1-r_u^1+k_u^1\\ \vdots \\ \mu_u^d-r_u^d+k_u^d \end{psmallmatrix},
\end{align}
it is immediate to conclude that the stochastic exponential
\begin{equation}\label{eq:stochExp}
    \mathcal E_0^t = \exp{-\int_0^t\theta_u^\intercal\ud W_u-\frac{1}{2}\int_0^t\theta_u^\intercal \theta_u\ud u}
\end{equation}
is a true martingale such that $W^\bQ$ defined in \eqref{eq:QBrow} is a Brownian motion according to the Girsanov theorem.

\subsection{Wealth Dynamics and Pricing under Different Measures}\label{sec:wealthDyn}
In what follows, we analyse the situation where collateral is exchanged in continuous time between the hedger and the counterparty to reduce possibly outstanding credit exposures.

Under the previously introduced probability space, we denote by $C$ an $\bF$-adapted stochastic process representing the exchanged collateral. 
$C^+$ represents the collateral received while $C^-$ is the collateral posted by the hedger. Collateral posted, resp. received, earns, resp. pays, an interest at a rate $r^{c,l}$, resp. $r^{c,b}$, and we set $C_T=0$, meaning that all collateral is returned at the terminal time when trading stops. The interest rates $r^{c,l}, r^{c,b}$ are assumed to be $\bF$-adapted and bounded processes. When collateral is posted via cash and the receiver is allowed to use it to fund own trading activity, then, according to equations 5.23 and 5.24 in \citet{bielecki2015valuation},  the wealth dynamics are as follows:
\begin{equation}\label{eq:wealth}
    \ud V_t(\varphi) = r_t V_t(\varphi)\ud t + \sum_{i=1}^{d}\xi_t^i(\ud S_t^i-r_t^i S_t^i\ud t+k_t^i S_t^i\ud t) + \ud \hat{F}_t^h+\ud A_t,
\end{equation}
where
\begin{equation}\label{eq:defF}
    \hat{F}_t^h \coloneqq \int_0^t C_u^+(r_u-r_u^{c,b})\ud u - \int_0^t C_u^-(r_u-r_u^{c,l})\ud u.
\end{equation}
We are then in a position to apply the representation stated in Section \ref{sec:relBSDE} so as to obtain pricing formulas under different measures.
\begin{assumption}\label{as:ass43}
    \item[a)] $A_t=p\mathbbm{1}_{[0,T]}(t)+X\mathbbm{1}_{[T]}(t)$ for $X\in \F_T=\F_T^S$, where by $\bF = \bF^S$ we mean that the filtration is generated by the risky assets, and $p$ is the initial price.
    \item[b)] The collateral process is of the form $C=c(V)$, where $c$ is a Lipschitz function.
\end{assumption}
Let us now apply Theorem \ref{th:mainTh}. We can substitute the $\P$ dynamics in \eqref{eq:Sdyn} inside \eqref{eq:wealth} as
\begin{align}
\begin{split}
    \ud V_t(\varphi) &= r_t V_t(\varphi)\ud t + \sum_{i=1}^{d}\xi_t^i(\mu_t^i S_t^i\ud t + S_t^i\left( \Sigma_t\bar{\rho}_t\ud W_t^\P \right)_i -r_t^i S_t^i\ud t+k_t^i S_t^i\ud t)\\
    &\qquad\qquad\qquad\qquad\qquad\qquad + \ud \hat{F}_t^h+\ud A_t
\end{split}
\\
\begin{split}
    &= r_t V_t(\varphi)\ud t + \left\langle \xi_t diag(S_t)\Sigma_t\bar{\rho}_t , \left( \Sigma_u\bar{\rho}_u \right)^{-1}\begin{psmallmatrix} \mu_u^1-r_u^1+k_u^1\\ \vdots \\ \mu_u^d-r_u^d+k_u^d \end{psmallmatrix}  \right\rangle \ud t \\
    &\qquad\qquad\qquad\qquad\qquad\qquad + \left\langle \xi_t diag(S_t)\Sigma_t\bar{\rho}_t , \ud W_t^\P \right\rangle + \ud \hat{F}_t^h+\ud A_t \; . 
\end{split}
\end{align}
We set $Z\coloneqq \xi diag(S)\Sigma \bar{\rho} \, \in \mathbb{R}^d$ and $\theta$ as in \eqref{eq:GirsanovKernelDiffusion},
therefore, we are left with a
BSDE providing a non-linear pricing rule, i.e.:
\begin{multline}\label{eq:wealthP}
\ud V_t(\varphi) = r_t V_t(\varphi)\ud t + Z_t^\intercal\theta_t\ud t + Z_t^\intercal\ud W_t^\P + (r_t-r_t^{c,b})C_t^+\ud t\\
- (r_t-r_t^{c,l})C_t^-\ud t+\ud A_t.
\end{multline}
Therefore, by  the results in Section \ref{sec:relBSDE}, we obtain the following version of Proposition 5.4 in \citet{bielecki2015valuation} under $\P$.

\begin{lemma}
Let $X$ be $\F_T^S$-measurable with $X\in\bL_T^2(\Omega,\F_T^S)$, then there exists a unique solution $(V,Z)\in\bS^2_T(\bR)\times\bH^2_T(\bR)$ to the BSDE \eqref{eq:wealthP}. The contract $A_t=p\mathbbm{1}_{[0,T]}(t)+X\mathbbm{1}_{[T]}(t)$ with collateral specification b) can be replicated by an admissible trading strategy $\xi$.  
We have $p_t=V_t$ and the process $V$ admits the following representation under $\P$:
\begin{multline}\label{eq:VunderP}
    p_t=V_t(\varphi) = -\bE^{\P} \left[ \frac{B_t}{B_T}X\mathcal E_t^T+ \int_t^T\frac{B_t}{B_u}\mathcal E_t^u \right.\\
    \left. \left[ (r_u-r_u^{c,b})C_u^+ -  (r_u-r_u^{c,l})C_u^- \right]\ud u |\F_t\right].
\end{multline}
\end{lemma}

\begin{remark}
    In the formula above, we recognize the term $\frac{B_t}{B_T}\mathcal E_t^T$ as the stochastic discount factor or state price density in the terminology of financial economics, see e.g. \cite{bookDuffie}, \cite{bookBarucciFontana}.
\end{remark}

Next, by exploiting the result derived in Theorem \ref{th:mainTh}, see eq.\eqref{eq:mainThEqChange}, we can equivalently rewrite the BSDE under the measure $\bQ$ as 
\begin{equation}\label{eq:wealthQ}
\ud V_t(\varphi) = r_t V_t(\varphi)\ud t + Z_t^\intercal\ud W_t^\bQ + (r_t-r_t^{c,b})C_t^+\ud t
- (r_t-r_t^{c,l})C_t^-\ud t+\ud A_t\;,
\end{equation}
allowing us to reformulate Proposition 5.4 in \citet{bielecki2015valuation}.

\begin{lemma}\label{lem:QBSDE}
Let $X$ be $\F_T^S$-measurable with $X\in\bL_T^2(\Omega,\F_T^S)$, then $\exists!$ solution $(V,Z)\in\bS^2_T(\bR)\times\bH^2_T(\bR)$ to the BSDE \eqref{eq:wealthQ}. The contract $A_t=p\mathbbm{1}_{[0,T]}(t)+X\mathbbm{1}_{[T]}(t)$ with collateral specification b) in Assumption \ref{as:ass43} can be replicated by an admissible trading strategy $\xi$. We have $p_t=V_t(\varphi)$ and the process $V_t$ admits the representation under $\bQ$:
\begin{equation}
    p_t=V_t(\varphi) = -\bE^{\bQ} \left[ \frac{B_t}{B_T}X+ \int_t^T\frac{B_t}{B_u} \left[ (r_u-r_u^{c,b})C_u^+ -  (r_u-r_u^{c,l})C_u^- \right]\ud u |\F_t\right].
\end{equation}
\end{lemma}

\begin{remark}
    A classical example for $X\in\bL_T^2(\Omega,\F_T^S)$ is given by a Lipschitz function of the risky assets $S^1,\ldots,S^d$ which covers e.g., European call and put options.
\end{remark}

\subsection{Benchmark approach interpretation}
In what follows we further analyse \eqref{eq:VunderP}, from the {\it Benchmark Approach} point of view. 
Let us explicitly state the form of the product of the ratio of cash-accounts times  the density process
\begin{align}\label{eq:stochExpDis}
    \frac{B_t}{B_T}\mathcal E_t^T &= \exp{-\int_t^T r_s\ud s -\frac{1}{2}\int_t^T\theta_s^\intercal \theta_s\ud s-\int_t^T\theta_s^\intercal\ud W^\P}\\
    &= \frac{\exp{\int_0^t r_s\ud s +\frac{1}{2}\int_0^t\theta_s^\intercal \theta_s\ud s+\int_0^t\theta_s^\intercal\ud W^\P}}{\exp{\int_0^T r_s\ud s +\frac{1}{2}\int_0^T\theta_s^\intercal \theta_s\ud s+\int_0^T\theta_s^\intercal\ud W^\P}}.
\end{align}
We recognize in this ratio the solution to the SDE satisfied by the \textit{Growth Optimal Portfolio} (GOP) in the benchmark approach of \citet{platen2006benchmark}, namely
\begin{equation}
    \frac{\ud S_t^{\delta^{\text{*}}}}{S_t^{\delta^{\text{*}}}} = r_t\ud t + \langle \theta_t, \theta_t\ud t + \ud W_t^\P \rangle.
\end{equation}
It is worth recalling that the GOP is the portfolio maximizing  the expected logarithmic utility. In \citet{platen2006benchmark}, the GOP plays a central role in financial evaluation. Indeed, it has the role of the numéraire for the physical probability measure $\P$, allowing for the derivation of a pricing formula for contingent claims under $\P$, i.e. the so-called \textit{real-world pricing formula}. 
The latter generalizes the risk-neutral valuation formula since it can be applied also when a risk-neutral measure does not exist. 
Exploiting our previous calculations, we deduce that \eqref{eq:VunderP} can be written as in the following corollary, hence extending the real world pricing formula derived in \citet{platen2006benchmark} in the presence of collateral and multiple interest rates.

\begin{corollary}
{[}The real World Pricing Formula - Case with collateral{]} The price of the claim $A_t=p\mathbbm{1}_{[0,T]}(t)+X\mathbbm{1}_{[T]}(t)$ with $X\in\bL_T^2(\Omega,\F_T^S)$ is given by
\begin{equation}
    p_t=V_t(\varphi) = -\bE^{\P} \left[ \frac{S_t^{\delta^{\text{*}}}}{S_T^{\delta^{\text{*}}}}X+ \int_t^T\frac{S_t^{\delta^{\text{*}}}}{S_u^{\delta^{\text{*}}}} \left[ (r_u-r_u^{c,b})C_u^+  -  (r_u-r_u^{c,l})C_u^- \right]\ud u |\F_t\right] \; ,
\end{equation}
where $S_t^{\delta^{\text{*}}}$ is the growth optimal portfolio with dynamics
\begin{equation}
    \frac{\ud S_t^{\delta^{\text{*}}}}{S_t^{\delta^{\text{*}}}} = r_t\ud t + \langle \theta_t, \theta_t\ud t + \ud W_t^\P \rangle.
\end{equation}
\end{corollary}

\subsection{Exchange option example}\label{sec:excOpt}
In this subsection we reconsider the classical example of an exchange option, see \citet{geman1995changes}, but in the presence of collateral and multiple interest rates. 
We define $\tilde{V}=\frac{V}{B}$ the discounted wealth process. From Equation \eqref{eq:wealthQ}, the dynamics under the measure $\bQ$ are given by 
\begin{equation}\label{eq:wealthQdiscExcOpt}
\ud \tilde{V}_t(\varphi) = \frac{Z_t^\intercal}{B_t}\ud W_t^\bQ + \frac{1}{B_t}\left( (r_t-r_t^{c,b})C_t^+  - (r_t-r_t^{c,l})C_t^- \right)\ud t + \frac{1}{B_t}\ud A_t.
\end{equation}

From Proposition \ref{prop:prop3_1} we know that the ratio $\frac{S^2}{B^2}$ is a $(\bQ,\bF)$-local martingale.
Such ratio can be used to construct a new probability measure. Let us first compute the dynamics of $\frac{B^2}{S^2}$ under $\bQ$ as 
\begin{align}\label{eq:invS}
\ud \left( \frac{B_t^2}{S_t^2}\right) ={}& \frac{B_t^2r_t^2\ud t}{S_t^2} + B_t^2 \left( - \frac{1}{(S_t^2)^2} \left(r_t^2S_t^2\ud t + S_t^2\left(\Sigma_t \bar{\rho}_t\ud W_t^\bQ \right)_2 \right.\right. \\
&\qquad\qquad\qquad\qquad\qquad\qquad\qquad\qquad \left.\left. + \frac{1}{2}\frac{2}{(S_t^2)^3}(S_t^2)^2(\sigma^2)^2\ud t\right)\right)\\
=& \frac{B_t^2}{S_t^2} \left( (\sigma_t^2)^2\ud t - \sigma_t^2\sum_{j=1}^d\bar{\rho}_{2,j}\ud W_t^{\bQ,j} \right) = \frac{B_t^2}{S_t^2} \left( (\sigma_t^2)^2\ud t - \sigma_t^2\bar{\rho}_{2,\cdot,t} \ud W_t^{\bQ} \right)  \; ,
\end{align} 
where $\left(\Sigma_t \bar{\rho}_t\ud W_t^\bQ \right)_2$ refers to the second coordinate of the vector $\left(\Sigma_t \bar{\rho}_t\ud W_t^\bQ \right)$ and $\bar{\rho}_{2,\cdot,t}$ is the second row of $\bar{\rho}_{t}$. We can write
\begin{align}\label{eq:invS2}
\ud \left( \frac{B_t^2}{S_t^2}\right) ={}& \frac{B_t^2}{S_t^2} \left( (\sigma_t^2)^2 \frac{\bar{\rho}_{2,\cdot,t}\bar{\rho}_{\cdot,2,t}}{\bar{\rho}_{2,\cdot,t}\bar{\rho}_{\cdot,2,t}} \ud t - \sigma_t^2\bar{\rho}_{2,\cdot,t}\ud W_t^{\bQ} \right) \\
={}& \frac{B_t^2}{S_t^2} \left( -\sigma_t^2\bar{\rho}_{2,\cdot,t}\left( \ud W_t^{\bQ} - \sigma_t^2\frac{\bar{\rho}_{\cdot,2,t}}{\bar{\rho}_{2,\cdot,t}\bar{\rho}_{\cdot,2,t}}\ud t\right)\right) \; \,,
\end{align} 
where we have $\bar{\rho}_{2,\cdot,t}\bar{\rho}_{\cdot,2,t}=\rho_{2,2}=1$,  since $\bar{\rho}$ is a Cholesky factorization of the correlation matrix.

\noindent Define
\begin{itemize}\label{eq:defQS2}
    \item $\frac{\partial \bQ^{S^2}}{\partial \bQ}\bigg{|}_{\F_T}\coloneqq \frac{S_T^2}{B_T^2}\frac{B_0^2}{S_0^2}  \; ;$
    \item $\ud W_t^{\bQ^{S^2}}\coloneqq \ud W_t^{\bQ} - \sigma_t^2\frac{\left(\bar{\rho}_{2,\cdot,t}\right)^\intercal}{\bar{\rho}_{2,\cdot,t}\bar{\rho}_{\cdot,2,t}}\ud t  \; .$
\end{itemize}
Since
\begin{multline}\label{eq:wealthQdisc}
\ud \tilde{V}_t(\varphi) = \frac{1}{B_t} \sum_{i=1}^d\sum_{j=1}^d\xi_t^1S_t^i\sigma_t^i\bar{\rho}_{i,j} \ud W_t^{\bQ,j} \\
+ \frac{1}{B_t}\left( (r_t-r_t^{c,b})C_t^+  - (r_t-r_t^{c,l})C_t^- \right)\ud t + \frac{1}{B_t}\ud A_t\;,
\end{multline}
then the dynamics of the product $\frac{V}{B}\frac{B^2}{S^2}$ is given by:
\begin{align}
    \ud \left( \frac{V}{B}\frac{B^2}{S^2} \right)_t & = \frac{B_t^2}{S_t^2} \left( \frac{1}{B_t} \sum_{i=1}^d\sum_{j=1}^d\xi_t^1S_t^i\sigma_t^i\bar{\rho}_{i,j} \ud W_t^{\bQ,j} \right.\\
&\qquad\quad \left. + \frac{1}{B_t}\left( (r_t-r_t^{c,b})C_t^+  - (r_t-r_t^{c,l})C_t^- \right)\ud t + \frac{1}{B_t}\ud A_t \right) \\
&\qquad\quad + \frac{V_t}{B_t}\frac{B_t^2}{S_t^2} \left( (\sigma_t^2)^2\ud t - \sigma_t^2\sum_{j=1}^d\bar{\rho}_{2,j}\ud W_t^{\bQ,j} \right)\\
&\qquad\quad +\ud \left\langle \sum_{i,j=1}^d\int_0^\cdot\xi_u^1S_u^i\sigma_u^i\bar{\rho}_{i,j,u} \ud W_u^{\bQ,j}, - \sum_{j=1}^d\int_0^\cdot\sigma_u^2\bar{\rho}_{2,j,u}\ud W_u^{\bQ,j} \right\rangle_t \frac{B_t^2}{B_t S_t^2} \; .
\end{align}
Concerning the quadratic covariation term we have
\begin{align}
& \ud \left\langle \sum_{i,j=1}^d\int_0^\cdot\xi_u^1S_u^i\sigma_u^i\bar{\rho}_{i,j,u} \ud W_u^{\bQ,j}, - \sum_{j=1}^d\int_0^\cdot\sigma_u^2\bar{\rho}_{2,j,u}\ud W_u^{\bQ,j} \right\rangle_t \frac{B_t^2}{B_t S_t^2} \\
& = -\sum_{i=1}^d\xi_t^1S_t^i\sigma_t^i\sigma_t^2 \sum_{j=1}^d\sum_{k=1}^d \bar{\rho}_{i,j,t}\bar{\rho}_{2,k,t} \ud\langle  W_\cdot^{\bQ,j},\ud W_\cdot^{\bQ,k} \rangle_t \frac{B_t^2}{B_t S_t^2}\\
& = -\frac{B_t^2}{B_t S_t^2} \sum_{i=1}^d\xi_t^1S_t^i\sigma_t^i\sigma_t^2 \sum_{k=1}^d \bar{\rho}_{i,k,t}\bar{\rho}_{2,k,t} \ud t\\
& = -\frac{B_t^2}{B_t S_t^2} \left( \xi_t diag (S_t) \Sigma_t\bar{\rho}_t \right) \sigma_t^2(\bar{\rho}_{2,\cdot,t})^\intercal \ud t \; ,
\end{align}
where we recall that $\xi_t diag (S_t) \Sigma_t\bar{\rho}_t=Z_t$ and $\bar{\rho}_{2,\cdot,t}$ denotes the second row of $\bar{\rho}_t$. Collecting all terms we have
\begin{align}
\begin{aligned}
\ud \left( \frac{V}{B}\frac{B^2}{S^2} \right)_t & = \frac{B_t^2}{S_t^2} \left( \frac{1}{B_t} Z_t^\intercal \ud W_t^{\bQ}  + \frac{1}{B_t}\left( (r_t-r_t^{c,b})C_t^+  - (r_t-r_t^{c,l})C_t^- \right)\ud t + \frac{1}{B_t}\ud A_t \right) \\
&\qquad\quad + \frac{V_t}{B_t}\frac{B_t^2}{S_t^2}  \left( -\sigma_t^2\bar{\rho}_{2,\cdot,t}\left( \ud W_t^{\bQ} - \sigma_t^2\bar{\rho}_{\cdot,2,t}\ud t\right)\right) -\frac{B_t^2}{B_t S_t^2} Z_t^\intercal \sigma_t^2(\bar{\rho}_{2,\cdot,t})^\intercal \ud t\;,
\end{aligned}\\
\end{align}
therefore, after rearranging, we obtain
\begin{align}
\begin{aligned}
\ud \left( \frac{V}{B}\frac{B^2}{S^2} \right)_t & = \frac{B_t^2}{S_t^2} \left[ \frac{1}{B_t} Z_t^\intercal \left(\ud W_t^{\bQ} -\sigma_t^2(\bar{\rho}_{2,\cdot,t})^\intercal \ud t\right) \right.\\
&\qquad\quad \left. + \frac{1}{B_t}\left( (r_t-r_t^{c,b})C_t^+  - (r_t-r_t^{c,l})C_t^- \right)\ud t+ \frac{1}{B_t}\ud A_t \right] \\
&\qquad\quad  - \frac{V_t}{B_t}\frac{B_t^2}{S_t^2}\sigma_t^2\bar{\rho}_{2,\cdot,t}  \left( \ud W_t^{\bQ} - \sigma_t^2(\bar{\rho}_{2,\cdot,t})^\intercal \ud t\right) \; .
\end{aligned}\\
\end{align}

If we now define $\tilde{Z}\coloneqq Z_t-V_t\sigma_t^2\bar{\rho}_{2,\cdot,t}$, then, by the Girsanov Theorem, we have:
\begin{multline}
\ud \left( \frac{V_t}{B_t}\frac{B_t^2}{S_t^2} \right) =  \frac{B_t^2}{S_t^2B_t} \tilde{Z}^\intercal\ud W_t^{\bQ^{S^2}} + \frac{B_t^2}{S_t^2B_t}\left[ \left((r_t-r_t^{c,b})C_t^+  - (r_t-r_t^{c,l})C_t^- \right)\ud t  + \ud A_t\right]   \; .
\end{multline}
We omit the conditional expectation representation of the BSDE above, since it is immediate. Now, under the further assumption $r^2=r$, by applying Theorem \ref{th:mainTh} we recover, the usual change of Numéraire, but in the presence of a non-linear, recursive integral term:
\begin{equation}
    p_t = - S_t^2\bE^{\bQ^{S^2}} \left[ \frac{X}{S_T^2}+ \int_t^T\frac{1}{S_s^2} \left[ (r_s-r_s^{c,b})C_s^+ -  (r_s-r_s^{c,l})C_s^- \right]\ud s |\F_t\right] \; .
\end{equation}
In the present example, $X$ is the payoff of the exchange option, i.e.
\begin{equation}
    X=(S_T^1-S_T^2)^+.
\end{equation}
In summary, we obtain
\begin{equation}
    p_t = - S_t^2\bE^{\bQ^{S^2}} \left[ \frac{(S_T^1-S_T^2)^+}{S_T^2}+ \int_t^T\frac{1}{S_s^2} \left[ (r_s-r_s^{c,b})C_s^+  -  (r_s-r_s^{c,l})C_s^- \right]\ud s |\F_t\right] \; ,
\end{equation}
i.e.
\begin{equation}
    p_t = - S_t^2\bE^{\bQ^{S^2}} \left[ \left(\frac{S_T^1}{S_T^2}-1\right)^+ + \int_t^T\frac{1}{S_s^2} \left[ (r_s-r_s^{c,b})C_s^+ -  (r_s-r_s^{c,l})C_s^- \right]\ud s \big|\F_t\right] \; .
\end{equation}
When there is no collateral, we recover the classical Margrabe's formula under the standard Black-Scholes multivariate model.

\subsection{Pure jump setting example}
We assume that the risky asset evolves now according to
\begin{equation}
    \frac{\ud S_t^i}{S_{t-}^i} = \mu^i\ud t+\int_\bR(e^z-1)(N-\nu^\P)(\ud t,\ud z)\; .
\end{equation}
which corresponds to the SDE associated to the class of exponential Lévy models.
For the sake of the present example, we assume that $N$ has compensator $\nu^\P(\ud z)\ud t = \lambda\mathbbm{1}_\alpha(\ud z) \ud t$, i.e. we consider a Poisson process with constant intensity $\lambda$ and fixed jump size $\alpha$. We construct the martingale measure along the following steps. In line with \citet{bielecki2015valuation}, our objective is to guarantee that the process 
\begin{equation}
    \tilde{S}_t^i \coloneqq \left( \frac{S^i}{B^i}\right)_{0\leq t\leq T}
\end{equation}
is a local martingale.
Consider next
\begin{align}
    \ud \tilde{S}_t^i & = \tilde{S}_{t-}^i\left( (\mu^i-r^i)\ud t +\int_\bR(e^z-1)(N-\nu^\P)(\ud t,\ud z)\right) \\
    & = \tilde{S}_{t-}^i\left( (e^\alpha-1)N(\ud t,\lbrace\alpha\rbrace) - \int_\bR(e^z-1)\lambda\mathbbm{1}_\alpha(z)\ud z\ud t)+(\mu^i-r^i)\ud t\right)\\
    & = \tilde{S}_{t-}^i\left( (e^\alpha-1)N(\ud t,\lbrace\alpha\rbrace) - \left((e^\alpha-1)\lambda-(\mu^i-r^i) \right)\ud t \right)\\
    & = \tilde{S}_{t-}^i\left( (e^\alpha-1)N(\ud t,\lbrace\alpha\rbrace) - (e^\alpha-1) \left(\lambda-\frac{\mu^i-r^i}{e^\alpha-1} \right)\ud t \right)\\
    & = \tilde{S}_{t-}^i\left( (e^\alpha-1)N(\ud t,\lbrace\alpha\rbrace) - (e^\alpha-1)\lambda \left(1-\frac{\mu^i-r^i}{(e^\alpha-1)\lambda} \right)\ud t \right) \; .
\end{align}
Recalling the notation of Theorem \ref{th:mainTh}, this means that we have
\begin{equation}
    \delta (z) = \delta = -\frac{\mu-t}{(e^\alpha -1)\lambda} \; ,
\end{equation}
which is the Girsanov kernel in the present setting. Under $\bQ$, we have
\begin{equation}
    \frac{\ud S_t^i}{S_{t-}^i} = r^i\ud t+\int_\bR(e^z-1)(N-\nu^\P)(\ud t,\ud z) \; ,
\end{equation}
with
\begin{align}
    \nu^\bQ (\ud z) & = (1+\delta (z))\nu^\P(\ud z)\\
    & = \left(1-\frac{\mu-r}{(e^\alpha-1)\lambda}\right) \lambda \mathbbm{1}_\alpha (\ud z) \; .
\end{align}
In summary we have:
\begin{description}
  \item[Under $\P$] 
    \begin{equation}
        \frac{\ud S_t^i}{S_{t-}^i} = \mu^i\ud t+\int_\bR(e^z-1)(N-\nu^\P)(\ud t,\ud z)\; , \; i.e.
    \end{equation}
    \begin{equation}
        S_t^i = S_0^i\exp{\mu t-\int_\bR (e^z-1-z)\nu(\ud z)t + \int_\bR z(N-\nu^\P)(\ud t, \ud z)}\; .
    \end{equation}
  \item[Under $\bQ$] 
      \begin{equation}
        \frac{\ud S_t^i}{S_{t-}^i} = r^i\ud t+\int_\bR(e^z-1)(N-\nu^\bQ)(\ud t,\ud z)\; , \; 
    \end{equation}
    \begin{equation}
        \nu^\bQ (\ud z) = (1+\delta (z))\nu^\P(\ud z) = \left(1-\frac{\mu-r}{(e^\alpha-1)\lambda}\right) \lambda \mathbbm{1}_\alpha (\ud z) \; .
    \end{equation}
    \item[Density process $\mathcal E_t^T$] 
    \begin{multline}
    \mathcal E_t^T = \exp\left\lbrace \int_t^T\int_{\bR^d}\ln(1+\delta(z)) \right.\\
    \left. -\delta(z)\nu^{\P}(\ud z) \ud r + \int_t^T\int_{\bR^d}\ln(1+\delta(z))\tilde{N}(\ud r,\ud z)\right\rbrace \, .
\end{multline}
\end{description}

\subsubsection{Portfolio dynamics with jumps}
We start again from the wealth dynamics
\begin{equation}
    \ud V_t(\varphi) = r_t V_t(\varphi)\ud t + \sum_{i=1}^{d}\xi_t^i(\ud S_t^i-r_t^i S_t^i\ud t+k_t^i S_t^i\ud t) + \ud \hat{F}_t^h+\ud A_t \; ,
\end{equation}
then, substituting the $\P$ dynamics of $S$, and setting $d=1$, we have
\begin{multline}
    \ud V_t(\varphi) = r_t V_t(\varphi)\ud t + \xi_t^1\left( S_{t-}^1 \left( \mu^1\ud t+\int_\bR(e^z-1)(N-\nu^\P)(\ud t,\ud z)\right) - r^1S^1\ud t\right) \\
    + \ud \hat{F}_t^h+\ud A_t \; ,
\end{multline}
it follows that, setting
\begin{equation}
    U^1 \coloneqq \xi^1S_{t-}^1
\end{equation}
we obtain
\begin{equation}
\begin{aligned}\label{eq:PBSDE}
    \ud V_t(\varphi) &= r_t V_t(\varphi)\ud t + U_t^1(\mu^1-r^1)\ud t + U_{t-}^1\int_\bR(e^z-1)(N-\nu^\P)(\ud t,\ud z) \\
    & \qquad\qquad\qquad\qquad\qquad\qquad\qquad\qquad\qquad\qquad\qquad + \ud \hat{F}_t^h+\ud A_t \\
    &= r_t V_t(\varphi)\ud t + \int_\bR U_{t-}^1\frac{\mu^1-r^1}{(e^\alpha-1)\lambda}(e^z-1)\nu^\P(\ud z) \ud t \\
    & \qquad\qquad\qquad\qquad\ +\int_\bR U_{t-}^1(e^z-1)(N-\nu^\P)(\ud t,\ud z) + \ud \hat{F}_t^h+\ud A_t \; ,
\end{aligned}
\end{equation}
allowing us to apply
Theorem \ref{th:mainTh} in the present pure jump setting. 
\begin{lemma}
Let $X$ be $\F_T^S$-measurable with $X\in\bL_T^2(\Omega,\F_T^S)$, then there exists a unique solution $(V,Z)\in\bS^2_T(\bR)\times\bH^2_{T,N}(\bR)$ to the BSDE \eqref{eq:PBSDE}. The contract $A_t=p\mathbbm{1}_{[0,T]}(t)+X\mathbbm{1}_{[T]}(t)$ with collateral specification b) in Assumption \ref{as:ass43}, can be replicated by an admissible trading strategy $\xi$. We have $p_t=V_t(\varphi)$ and the process $V$ admits the representation under $\P$.
\begin{multline}\label{eq:VunderPjump}
    p_t=V_t(\varphi) = -\bE^{\P} \left[ \frac{B_t}{B_T}X\mathcal E_t^T + \int_t^T \frac{B_t}{B_u}\mathcal E_t^u \right.\\
    \left. + \left[ (r_u-r_u^{c,b})C_u^+ -  (r_u-r_u^{c,l})C_u^- \right]\ud u |\F_t\right].
\end{multline}
\end{lemma}
Next,  again using the change of measure formula of Theorem \ref{th:mainTh}, we can rewrite the BSDE under the measure $\bQ$ as 
\begin{equation}
    \ud V_t(\varphi) = r_t V_t(\varphi)\ud t + \int_\R U_{t-}^1(e^z-1)(N-\nu^\P)(\ud t,\ud z)  + \hat{F}_t^h+\ud A_t \; .
\end{equation}
Finally, we can reformulate Proposition 5.4 in \citet{bielecki2015valuation} in a pure jump setting. We omit the statement since it is analogous to our Lemma \ref{lem:QBSDE}. 

\subsection{A generalization of the change of numéraire technique}
The reasoning we illustrated in Section \ref{sec:wealthDyn} provides us with a generalization of the well-known change of numéraire technique that covers recursive conditional expectations.
The technique involves the following steps:
\begin{enumerate}
    \item Start by considering wealth dynamics under a starting probability measure $\bQ$, for example as in Equation \ref{eq:wealthQ} for $\tilde{V}$;
    \item Consider a traded asset, suitably discounted so as to treat a $(\bQ,\bF)$-martingale (the ratio $\frac{S^2}{B^2}$ in Section \ref{sec:excOpt}) and compute its (inverse) dynamics via the It\^{o} formula;
    \item Identify the stochastic drivers under the new measure $\bQ^\prime$;
    \item Compute the dynamics of the product between $\tilde{V}$ and the (inverse) density process;
    \item Apply Theorem \ref{th:mainTh} to obtain the expression for the price under the new probability measure $\bQ^\prime$.
\end{enumerate}

\section*{Acknowledgement}
We are grateful to Samuel Cohen for drawing our attention to useful references.
\bibliographystyle{plainnat}
\bibliography{BiblioBsde}

\end{document}